\documentclass[final,reqno]{amsart}

\usepackage{graphicx}
\usepackage{amsmath,amsthm,amsfonts,amscd,amssymb}
\usepackage[color,notref,notcite]{showkeys}
\usepackage{url}
\usepackage{subeqnarray}
\usepackage[color,notref,notcite]{showkeys} 
\definecolor{labelkey}{rgb}{0.6,0,1}
\usepackage{enumitem}
\usepackage{algorithm}
\usepackage{algpseudocode}
\usepackage{citesort}

\theoremstyle{plain}
\newtheorem{theorem}{Theorem}[section]

\newtheorem{lemma}[theorem]{Lemma}

\theoremstyle{definition}
\newtheorem{definition}[theorem]{Definition}

\def\bhyp#1{\begin{equation}\label{#1}\begin{array}{c}}
\def\ehyp{\end{array}\end{equation}}

\newcounter{cst}

\theoremstyle{remark}

\numberwithin{equation}{section}
\numberwithin{figure}{section}

\newcommand{\RR}{{\mathbb R}}
\newcommand{\NN}{{\mathbb N}}

\def\O{\Omega}
\def\dsp{\displaystyle}

\def\disc{{\mathcal D}}

\def\dr{\partial}

\newcommand{\x}{\pmb{x}}
\newcommand{\cA}{{\mathcal A}}
\newcommand{\cB}{{\mathcal B}}

\def\bvarphi{\boldsymbol{\phi}}

\def\half{\frac{1}{2}}
\def\g{\mathbf{g}}
\def\H{\mathbf{H}}
\def\w{\mathbf{w}}
\def\L{\mathbf{L}}
\def\X{\mathbf{X}}
\def\Y{\mathbb{Y}}

\newif\ifcorr\corrtrue
\usepackage[normalem]{ulem}
\definecolor{violet}{rgb}{0.580,0.,0.827}

\def\bpsi{{\boldsymbol \psi}}
\def\u{{\boldsymbol u}}

\def\v{{\boldsymbol v}}

\newcommand{\ud}{\, \mathrm{d}} 
\def\div{\mathop{\rm div}}

\title[Analysis of schemes for the Navier-Stokes Equation]{A generic Scheme For the time-dependent Navier-Stokes Equation Coupled With The Heat Equation}

\author{Yahya Alnashri}
\address[Yahya Alnashri]{Department of Mathematics, Al-Qunfudah University College, Umm Al-Qura University, Saudi Arabia}
\email{yanashri@uqu.edu.sa}

\keywords{Navier-Stokes problem, heat equation, time-dependent problem, gradient discretisation method, gradient schemes, finite volume scheme, convergence analysis}

\date{\today}

\begin{document}
\newcommand{\subscript}[2]{$#1 _ #2$}

\begin{abstract}
In this work, we study the gradient discretisation method (GDM) of the time-dependent Navier-Stokes equations coupled with the heat equation, where the viscosity depends on the temperature. We design the discrete method and prove its convergence without non-physical conditions. The paper is closed with numerical experiments that confirm the theoretical results.
\end{abstract}

\maketitle


\section{Introduction}
Navier-Stokes problems are of significant importance for a variety of applications. Examples are cooling processes, industrial furnaces, boilers, and heat exchangers \cite{B2}. In this work, we extend the gradient discretisation method (GDM) to the transient coupled Navier-Stokes and heat equations describing the flow of a viscous incompressible fluid. The GDM introduced in \cite{B1} is an abstract setting that covers different families of numerical schemes, and it has been applied successfully for various problems in porous media. The model we consider here is governed by the unknowns $\bar \u$, $\bar p$, and $\bar S$ satisfying the following equations:
\begin{subequations}\label{nsp}
\begin{align}
\dr_t \bar\u-\div(V(\bar S)\nabla\bar \u)+(\bar \u\cdot \nabla\bar \u)+\nabla\bar p=\g &\mbox{\quad in $\O \times (0,T)$,} \label{nsp-1}\\
\div \bar \u=0 &\mbox{\quad in $\O \times (0,T)$,} \label{nsp-2}\\
\dr_t \bar S-\mu\Delta\bar S + (\bar \u\cdot\nabla)\bar S=h &\mbox{\quad in $\O \times (0,T)$,} \label{nsp-3}\\
\bar \u=0 &\mbox{\quad on $\dr\O \times (0,T)$,} \label{nsp-4}\\
\bar S=0&\mbox{\quad on $\dr\O \times (0,T)$}\label{nsp-5},\\
\bar\u(\x,0)=\u_0 &\mbox{\quad in $\O$},\label{nsp-5}\\
\bar S(\x,0)=S_0 &\mbox{\quad in $\O$}.\label{nsp-6}
\end{align}
\end{subequations}

Here, the vector $\bar \u$ and the variables $\bar p$ and $\bar S$ denote the velocity, the pressure and the temperature of the fluid, respectively, and they are defined on the polygonal domain $\O\subset \RR^d$ $(d=2,3)$. The functions $\g$ and $h$ denote the distributions, and $\u_0$ and $S_0$ denote the initial data and temperature inside the body $\O$. Note that the model involves non-linear viscosity ($V$) depending on the temperature. We focus here on homogeneous Dirichlet boundary conditions for simplicity, whereas our analysis can easily be extended to different conditions.

It is proved in \cite{REF1} that the model \eqref{nsp} has a weak solution under standard assumptions on data. Because of the non-linearity term, it is not directly used to compute analytical solutions. Therefore, there is a high demand for numerical approximation to find a solution in practice. A large amount of articles are available on discretisation of Navier-Stokes equations. First, the problems without the heat equation have been studied extensively in different works; see, for example, \cite{EYMARD-1,EYMARD2,A3,A4}. 

Recently, several studies have started analysing stationary coupled Navier-Stokes and heat equations. Without exhaustivity, we cite that problems are discretised by a spectral method in \cite{A8} and a finite element method in \cite{A9}. Using the GDM, \cite{Y-NSP-2021} provides a complete convergence analysis. In \cite{REF4}, virtual element discretisation is employed to obtain an order of convergence.

Although several works on the numerical approximation of unsteady Navier-Stokes coupled with heat equations have been carried out in \cite{REF1,A10,A11}- to our knowledge- there is no available literature on the numerical analysis of non-conforming discretisation of the model \eqref{nsp}. This paper aims to utilise the gradient discretisation method to propose a generic scheme and a new convergence analysis approach for the model \eqref{nsp}.

The rest of the paper is organised as follows: In Section \ref{sec-2}, we state our assumptions and define the continuous weak formulation of our model. Also, we introduce the fully discrete scheme and some properties. In Section \ref{sec-3}, we prove the proposed method's convergence in space and time variables under the physical hypothesis on data. Section \ref{sec-numerical} is devoted to numerical experiments to justify the theoretical results.


\section{Continuous and Discrete Problems}\label{sec-2}
Let $\O$ be a connected open subset of $\RR^d\; (d=1,2,3)$ with a boundary $\dr\O$. We use the notation $\L^2(\O)^d$ for the space of square-integrable vector functions and $\H^1(\O)^d$ for the Sobolev space of vector functions. For $T>0$, we denote by $L^2(0,T;\H^1(\O)^d)$, the space of vector functions $\v$ such that for almost all $t\in[0,T]$, $\v(\cdot,t) \in \H^1(\O)$, and it is endowed with norm defined by
\[
\|\v\|_{\L^2(0,T;\H^1(\O)^d)}=\Big(\dsp\int_0^T\|\v\|_{\H^1(\O)^d} \ud t\Big)^{1/2}.
\]
For scalar functions, we use the notations $L^2(\O)$ and $L^2(0,T;H^1(\O))$. From now on, let$''\cdot''$ be the dot product on $\RR^d$, if $\bvarphi=(\varphi_{ij})_{i,j=1,...,d}$ and $\bpsi=(\psi_{ij})_{i,j=1,...,d} \in \RR^{d \times d}$, $\bvarphi:\bpsi=\sum_{i,j=1}^d \varphi_{ij}\psi_{ij}$ is the doubly contracted product on $\RR^{d \times d}$, and
\begin{equation}\label{eq-A}
\cA(\u,\v,\w)=\dsp\sum_{i,j=1}^d\dsp\int_\O u_i(\dr_i v_j)w_j \ud \x, \quad \forall (\u,\v,\w) \in \H_0^1(\O)^d \times \H_0^1(\O)^d \times \H_0^1(\O)^d,
\end{equation}
\begin{equation}\label{eq-B}
\cB(\u,r,S)=\dsp\sum_{i=1}\dsp\int_\O u_i(\dr_i r)S \ud \x, \quad \forall (\u,r,S) \in \H_0^1(\O)^d \times H_0^1(\O) \times H_0^1(\O).
\end{equation}  

We assume the model data satisfies the following conditions:
\begin{equation}\label{assump-nsp}
\left.
\begin{aligned}
&\mu \mbox{ is a positive constant,}\\
&\mbox{$V \in L^\infty(\RR)$, and it holds, for two positive constants $a_1$ and $a_2$},\\
&\mbox{for $\xi \in \RR$, $a_1 \leq V(\xi) \leq a_2$},\\
&\g \in L^2(0,T;\L^2(\O)^d) \mbox{ and } h \in L^2(0,T;L^2(\O)),\\
&\u_0 \in \H_0^1(\O)^d \mbox{ and } S_0 \in H_0^1(\O).
\end{aligned}
\right.
\end{equation}

Multiplying the problem \eqref{nsp} with trail functions $\v$, $q$, and $r$ and integrating by parts, the problem can be rewritten in a weak sense as the variational formulation. Under the above assumptions, the weak solution to \eqref{nsp} is to find $(\bar\u,\bar p,\bar S) \in L^2(0,T;\H_0^1(\O)^d) \times L^2(0,T;L_0^2(\O)) \times L^2(0,T;H_0^1(\O))$, such that,
\begin{subequations}\label{nsp-weak}
\begin{equation}\label{nsp-weak-1}
\begin{aligned}
&\mbox{for all $\v \in L^2(0,T;\H_0^1(\O)^d)$ such that $\dr_t \v \in L^2(\O\times (0,T))$ and $\v(\cdot,T)=0$}\\
&-\dsp\int_0^T\dsp\int_\O \bar\u(\x,t)\dr_t \v \ud \x \ud t
-\dsp\int_\O \u_0 \v(\x,0) \ud \x\\
&+\dsp\int_0^T\dsp\int_\O V(\bar S)(\x,t)\nabla\bar \u(\x,t) : \nabla\v(\x) \ud \x \ud t\\
&+\dsp\int_0^T\cA(\bar\u(\x,t),\bar\u(\x,t),\v(\x)) \ud t
-\dsp\int_\O(\div \v)(\x) \bar p(\x,t) \ud \x\\
&=\dsp\int_0^T\dsp\int_\O \g(\x,t) \v(\x) \ud \x \ud t,
\end{aligned}
\end{equation}
\begin{equation}\label{nsp-weak-2}
\mbox{for all } q \in L^2(0,T;L_0^2(\O)), \quad
-\dsp\int_0^T\dsp\int_\O \div \bar\u(\x,t) q(\x) \ud \x \ud t =0,
\end{equation}
\begin{equation}\label{nsp-weak-3}
\begin{aligned}
&\mbox{for all $r \in L^2(0,T;H_0^1(\O))$ such that $\dr_t r \in L^2(\O\times (0,T))$ and $r(\cdot,T)=0$}\\
&\quad-\dsp\int_0^T\dsp\int_\O \bar S(\x,t)\dr_t r(\x) \ud \x \ud t
-\dsp\int_\O S_0 r(\x,0) \ud \x\\
&\quad+\mu \dsp\int_0^T\dsp\int_\O \nabla\bar S(\x,t)\cdot \nabla r(\x) \ud \x \ud t
+\dsp\int_0^T\cB(\bar\u(\x,t),\bar S(\x,t),r(\x)) \ud t\\
&=\dsp\int_0^T\dsp\int_\O h(\x,t)r(\x) \ud \x \ud t, 
\end{aligned}
\end{equation}
\end{subequations}
where $L_0^2(\O):=\{ q\in L^2(\O)\;:\; \dsp\int_\O q(\x) \ud \x=0 \}$, and $\cA$ and $\cB$ are respectively defined by \eqref{eq-A} and \eqref{eq-B}.

Under the above standard assumptions on data, \cite{REF1} shows that there exists at least one solution to the problem \eqref{nsp-weak}. Now, let us construct the gradient discretisation method for our model. For the sake of completeness, we reintroduce the discrete elements (defined in \cite[Definition 2.2]{Y-NSP-2021} ) to discretise the spatial space $\O$.

\begin{definition}\label{def-gd-nsp}
A gradient discretisation $\disc$ for the Navier--Stokes model coupled with the heat equation  is given by $\disc=(\X_{\disc,0},Y_\disc,  \Y_{\disc,0}, \Pi_\disc, \widetilde\Pi_\disc, \nabla_\disc, \widetilde\nabla_\disc, \div_\disc, \chi_\disc, \cA_\disc, \cB_\disc)$, where:
\begin{itemize}
\item $\X_{\disc,0}$, $Y_\disc$ and $\Y_{\disc,0}$ are finite-dimensional vector spaces over $\RR^d$,
\item $\Pi_\disc : \X_{\disc,0} \to \L^2(\O)$ and $\widetilde\Pi_\disc : \Y_{\disc,0} \to L^2(\O)$ are the function reconstruction operators, which are linear,
\item $\nabla_\disc : \X_{\disc,0} \to \L^2(\O)^d$ and $\widetilde\nabla_\disc : \Y_{\disc,0} \to L^2(\O)^d$ are the gradient reconstruction operators, which are linear and must be defined so that $\| \nabla_\disc\cdot \|_{\L^2(\O)^d}$ and $\| \widetilde\nabla_\disc\cdot \|_{L^2(\O)^d}$ are norms on $\X_{\disc,0}$ and $\Y_{\disc,0}$, respectively,
\item $\chi_\disc:Y_\disc \to L^2(\O)$ is the reconstruction of the approximate pressure, and must be defined so that $\| \chi_\cdot \|_{L^2(\O)}$ is a norm on $Y_\disc$. Define the quantity $B_\disc >0$ by
\begin{equation}\label{new-constant}
B_\disc=\min_{q\in Y_{\disc,0}-\{0\}}\max_{v\in \X_{\disc,0}-\{0\}}\dsp\frac{\dsp\int_\O \chi_\disc q {\div}_{\disc} v \ud \x}{ \| \nabla_\disc v \|_{\L^2(\O)^d} \| \chi_\disc q \|_{L^2(\O)} },
\end{equation} 
where $Y_{\disc,0}:=\{ q\in Y_\disc\; : \; \dsp\int_\O \chi_\disc q \ud \x =0   \}$.
\item $\div_\disc:\X_{\disc,0}\to \L^2(\O)$ is the discrete divergence operator,

\item $\cA_\disc:\X_{\disc,0}^2 \to \RR$ and $\cB_\disc:\X_{\disc,0}\times \Y_{\disc,0}\times \Y_{\disc,0} \to \RR$ are the discrete convection terms, and they must be defined such that
\begin{itemize}
\item $\cA_\disc$ and $\cB_\disc$ are continuous,
\item $\forall (\u,\u)\in \X_{\disc,0}^2$, $\cA_\disc(\u,\u)\geq 0$, and $\forall (\u,S,r)\in \X_{\disc,0}\times \Y_{\disc,0}\times \Y_{\disc,0}$, $\cB_\disc(\u,S,S)\geq 0$,
\item $\exists\; \alpha_\disc,\beta_\disc >0$ such that 
\[
|\cA_\disc(\u,\v)|\leq \alpha_\disc\| \nabla_\disc \u \|_{\L^2(\O)^d} \| \nabla_\disc \v \|_{\L^2(\O)^d},\mbox{ and }
\]
\[
|\cB_\disc(\u,S,r)|\leq \beta_\disc\| \nabla_\disc \u \|_{\L^2(\O)^d} \| \widetilde\nabla_\disc S \|_{L^2(\O)^d} \| \widetilde\nabla_\disc r \|_{L^2(\O)^d},
\]
\item $\cA_\disc(\u,\v)$ and $\cB_\disc(\u,S,r)$ are linear respect to $\v$ and $r$ respectively.
\end{itemize} 
\end{itemize}
\end{definition}
 
\begin{definition}\label{def-gd-nsp-T}
A space-time gradient discretisation $\disc_T$ for the time-dependent Navier-Stokes problem with the heat equation is given by $\disc_T=(\disc,J_\disc, \widetilde J_\disc,(t^{(n)})_{n=0,...,N})$, where:
\begin{itemize}
\item $\disc:=(\X_{\disc,0},Y_\disc,  \Y_{\disc,0}, \Pi_\disc, \widetilde\Pi_\disc, \chi_\disc \nabla_\disc, \widetilde\nabla_\disc, \div_\disc,\cA_\disc, \cB_\disc)$ is the gradient discretisation of the spatial space $\O$ defined in Defintion \ref{def-gd-nsp}.  
\item $J_\disc:\L^2(\O)^d\to \X_{\disc,0}$ and $\widetilde J_\disc:L^2(\O)\to \Y_{\disc,0}$ are interpolant operators.
\item $t^{(0)} =0<t^{(1)} <...<t^{(N)} =T$.
\end{itemize}
\end{definition}

Now, if $\disc_T$ is a space-time gradient discretisation defined in Definition \ref{def-gd-nsp-T}. Then the corresponding gradient scheme for \eqref{nsp-weak} is given by:
\begin{subequations}\label{nsp-gs}
\begin{equation*}
\begin{aligned}
&\u=(\u^{(n)})_{n=0,...,N}\in \X_{\disc,0}^{N+1}, p=(p^{(n)})_{n=0,...,N}\in Y_\disc^{N+1}, S=(S^{(n)})_{n=0,...,N}\in \Y_{\disc,0}^{N+1} ,\\
&\mbox{such that $\u^{(0)}=J_\disc \u_0$ and $\widetilde J_\disc S^{(0)}=S_0$ and for all $n=0,..,N-1$,}
\end{aligned}
\end{equation*}
\begin{equation}\label{nsp-gs-1}
\begin{aligned}
&\dsp\int_\O \delta_\disc^{(n+\frac{1}{2})}\u \cdot \Pi_\disc \v \ud \x
+\dsp\int_\O V(\widetilde\Pi_\disc S^{(n+1)})\nabla_\disc \u^{(n+1)}: \nabla_\disc \v \ud \x\\
&+\cA_\disc(\u^{(n+1)},\v)
-\dsp\int_\O({\div}_\disc \v) \chi_\disc p^{(n+1)} \ud \x\\
&=\dsp\int_\O \g \Pi_\disc \v \ud \x,
\quad \mbox{for all $\v\in \X_{\disc,0}$},
\end{aligned}
\end{equation}
\begin{equation}\label{nsp-gs-2}
\dsp\int_\O {\div}_\disc \u^{(n+1)} \chi_\disc q \ud x =0, \quad \mbox{for all $q \in Y_\disc$},
\end{equation}
\begin{equation}\label{nsp-gs-3}
\begin{aligned}
&\dsp\int_\O \widetilde\delta_\disc^{(n+\frac{1}{2})}S \cdot \widetilde\Pi_\disc r \ud \x
+\mu \dsp\int_\O \widetilde\nabla_\disc S^{(n+1)}\cdot \widetilde\nabla_\disc r \ud \x
+\cB_\disc(\u^{(n+1)},S^{(n+1)},r)\\
&=\dsp\int_\O h\widetilde\Pi_\disc r \ud \x,
\quad \mbox{for all $r \in \Y_{\disc,0}$},
\end{aligned}
\end{equation}
\end{subequations}
where $\cA_\disc$ and $\cB_\disc$ are the discrete convections described in Definition \ref{def-gd-nsp}. The gradient scheme problem \eqref{nsp-gs} admits at least one solution. At any time step $(n + 1)$, we note that a gradient scheme for a stationary problem is solved. Applying the same reasoning as in \cite[Section 2]{A8}, we can easily derive the existence of a solution to the elliptic discrete problem due to the finite dimension of discrete spaces.


Below, we introduce some properties of the space-time gradient discretisation to guarantee the convergence of corresponding gradient schemes. These properties are slight adaptations of the ones stated in \cite{EYMARD-1} to deal with the new discrete elements introduced in our gradient discretisation.

\begin{definition}
A sequence of a space-time gradient discretisation $(\disc_{T_m})_{m\in\NN}$ is said to be 
\begin{itemize}
\item coercive (resp. limit-conforming, resp. trilinear limit conforming, resp. compact) if its spatial component $(\disc_m)_{m\in\NN}$ is coercive in the sense of \cite[Definition 2.4]{Y-NSP-2021} (resp. limit-conforming in the sense of \cite[Definition 2.5]{Y-NSP-2021}, resp. trilinear limit-conforming in the sense of \cite[Definition 2.6]{Y-NSP-2021}, resp. compact in the sense of \cite[Definition 2.7]{Y-NSP-2021}).
\item consistent if 
\begin{enumerate}
\item its spatial component $(\disc_m)_{m\in\NN}$ is consistent in the sense of \cite[Definition 2.5]{Y-NSP-2021},
\item for all $\v\in \L^2(\O)^d$ and for all $v\in L^2(\O)$, $J_{\disc_m}\v \to \v$ in $\L^2(\O)^d$ and $\widetilde J_{\disc_m}v \to v$ in $L^2(\O)$, as $m\to\infty$,
\item $\delta t_{\disc_m} \to 0$, as $m\to\infty$.
\end{enumerate}
 \end{itemize}
\end{definition}


\section{Discrete Energy Estimates And Main Results}\label{sec-3}
First, we derive some estimates on discrete solutions in order to establish our main results. From now on, let the quantity $C_\disc$ be given by
\begin{equation}
\begin{aligned}
C_\disc = \max_{\v\in \X_{\disc,0}-\{0\}}\dsp\frac{\|\Pi_\disc \v \|_{\L^2(\O)}}{\| \nabla_\disc \v \|_{\L^2(\O)^d}}
&+\max_{v\in \Y_{\disc,0}-\{0\}}\dsp\frac{\|\widetilde\Pi_\disc v \|_{L^2(\O)}}{\| \widetilde\nabla_\disc v \|_{L^2(\O)^d}}\\
&+\max_{v\in Y_{\disc,0}-\{0\}}\dsp\frac{\|\div_\disc v \|_{L^2(\O)}}{\| \nabla_\disc v \|_{L^2(\O)^d}}, 
\end{aligned}
\end{equation}
which connects to the discrete Poincar\'e inequality.
 
\begin{lemma}\label{lemma-est-sol}
Assume that the conditions \eqref{assump-nsp} hold and  $(\u,p,S)$ is a solution to the discrete problem \eqref{nsp-weak}. Then there exist constants $C_1,C_2>0$ not depending on the discrete solution, such that, for all $n=0,...,N$,
\begin{equation}\label{eq-est-sol-1}
\| \Pi_\disc \u \|_{L^\infty(0,T;\L^2(\O))}^2
+\mu \| \nabla_\disc \u \|_{\L^2(\O\times(0,T))^d}^2
\leq C_1 + \|\Pi_\disc J_\disc \u_0\|_{\L^2(\O)}^2,
\end{equation}
and
\begin{equation}\label{eq-est-sol-2}
\| \widetilde\Pi_\disc S \|_{L^\infty(0,T;L^2(\O))}^2
+\| \widetilde\nabla_\disc S \|_{L^2(\O\times(0,T))^d}^2
\leq C_2 + \|\widetilde\Pi_\disc \widetilde J_\disc S_0\|_{L^2(\O)}^2.
\end{equation}
\end{lemma}

\begin{proof}
In \eqref{nsp-gs-1} and \eqref{nsp-gs-3}, let $\v:=\delta t^{n+\frac{1}{2}}\u^{(n+1)}$ and $r:=\delta t^{n+\frac{1}{2}}S^{(n+1)}$ to attain
\[
\begin{aligned}
&\dsp\int_\O \Big(\Pi_\disc \u^{(n+1)}-\Pi_\disc \u^{(n)}\Big) \cdot \Pi_\disc \u^{(n+1)} \ud \x
+a_1\dsp\int_{t^{(n)}}^{t^{(n+1)}}\dsp\int_\O | \nabla_\disc \u^{(n+1)} |^2 \ud \x \ud t\\
&+\dsp\int_{t^{(n)}}^{t^{(n+1)}}\cA_\disc(\u^{(n+1)},\u^{(n+1)})\ud t
-\dsp\int_{t^{(n)}}^{t^{(n+1)}}\dsp\int_\O({\div}_\disc \u^{(n+1)}) \chi_\disc p^{(n+1)} \ud \x \ud t\\
&\leq\dsp\int_{t^{(n)}}^{t^{(n+1)}}\dsp\int_\O \g \Pi_\disc \u^{(n+1)} \ud \x \ud t,
\end{aligned}
\]
and
\[
\begin{aligned}
&\dsp\int_\O \Big(\widetilde\Pi_\disc S^{(n+1)}-\widetilde\Pi_\disc S^{(n)}\Big) \cdot \widetilde\Pi_\disc S^{(n+1)} \ud \x
+\mu\dsp\int_{t^{(n)}}^{t^{(n+1)}}\dsp\int_\O | \widetilde\nabla_\disc S^{(n+1)} |^2 \ud \x \ud t\\
&+\dsp\int_{t^{(n)}}^{t^{(n+1)}}\cB_\disc(S^{(n+1)},S^{(n+1)},S^{(n+1)})\ud t\\
&=\dsp\int_{t^{(n)}}^{t^{(n+1)}}\dsp\int_\O h \widetilde\Pi_\disc S^{(n+1)} \ud \x \ud t.
\end{aligned}
\]
Taking $q=p$ implies $\int_{t^{(n)}}^{t^{(n+1)}}\int_\O({\div}_\disc \u^{(n+1)}) \chi_\disc p^{(n+1)} \ud \x \ud t=0$. Utilizing the positivity of terms $\cA_\disc(\u^{(n+1)},\u^{(n+1)})$ and $\cB_\disc(S^{(n+1)},S^{(n+1)},S^{(n+1)})$, it follows that
\[
\begin{aligned}
&\dsp\int_\O \Big(\Pi_\disc \u^{(n+1)}-\Pi_\disc \u^{(n)}\Big) \cdot \Pi_\disc \u^{(n+1)} \ud \x
+a_1\dsp\int_{t^{(n)}}^{t^{(n+1)}}\dsp\int_\O | \nabla_\disc \u^{(n+1)} |^2 \ud \x \ud t\\
&\leq\dsp\int_{t^{(n)}}^{t^{(n+1)}}\dsp\int_\O \g \Pi_\disc \u^{(n+1)} \ud \x \ud t,
\end{aligned}
\]
and
\[
\begin{aligned}
&\dsp\int_\O \Big(\widetilde\Pi_\disc S^{(n+1)}-\widetilde\Pi_\disc S^{(n)}\Big) \cdot \widetilde\Pi_\disc S^{(n+1)} \ud \x
+\mu\dsp\int_{t^{(n)}}^{t^{(n+1)}}\dsp\int_\O | \widetilde\nabla_\disc S^{(n+1)} |^2 \ud \x \ud t\\
&\leq\dsp\int_{t^{(n)}}^{t^{(n+1)}}\dsp\int_\O h \widetilde\Pi_\disc S^{(n+1)} \ud \x \ud t.
\end{aligned}
\]
From the inequality $(a-b)\cdot a \geq \dsp\frac{1}{2}(|a|^2+|b|^2)$, we conclude
\[
\begin{aligned}
&\frac{1}{2}\dsp\int_\O \Big[ |\Pi_\disc \u^{(n+1)}|^2 - |\Pi_\disc \u^{(n)}|^2 \Big]\ud \x 
+a_1\dsp\int_{t^{(n)}}^{t^{(n+1)}}\dsp\int_\O | \nabla_\disc \u^{(n+1)} |^2 \ud \x \ud t\\
&\leq\dsp\int_{t^{(n)}}^{t^{(n+1)}}\dsp\int_\O \g \Pi_\disc \u^{(n+1)} \ud \x \ud t,
\end{aligned}
\]
and
\[
\begin{aligned}
&\frac{1}{2}\dsp\int_\O \Big[ |\widetilde\Pi_\disc S^{(n+1)}|^2 - |\widetilde\Pi_\disc S^{(n)}|^2 \Big]\ud \x 
+\mu\dsp\int_{t^{(n)}}^{t^{(n+1)}}\dsp\int_\O | \widetilde\nabla_\disc S^{(n+1)} |^2 \ud \x \ud t\\
&\leq\dsp\int_{t^{(n)}}^{t^{(n+1)}}\dsp\int_\O h \widetilde\Pi_\disc S^{(n+1)} \ud \x \ud t.
\end{aligned}
\]
Let $0 \leq n \leq N$. Summation of the above formulations over $n\in\{0,...,m-1\}$, we get, for some $m\in\{0,...,N\}$,
\[
\begin{aligned}
&\frac{1}{2}\| \Pi_\disc \u^{(m)} \|_{\L^2(\O)}^2  
+a_1\| \nabla_\disc \u^{(m)} \|_{\L^2(\O\times(0,t^{(m)})^d}^2\\
&\leq\dsp\int_0^{t^{(m)}}\dsp\int_\O \g \Pi_\disc \u \ud \x \ud t
+\frac{1}{2}\| \Pi_\disc J_\disc \u_0 \|_{\L^2(\O)}^2,
\end{aligned}
\]
and
\[
\begin{aligned}
&\frac{1}{2}\| \widetilde\Pi_\disc S^{(m)} \|_{L^2(\O)}^2  
+\mu\| \widetilde\nabla_\disc S^{(m)} \|_{L^2(\O\times(0,t^{(m)})^d}^2\\
&\leq\dsp\int_0^{t^{(m)}}\dsp\int_\O h \widetilde\Pi_\disc S \ud \x \ud t
+\frac{1}{2}\| \widetilde\Pi_\disc \widetilde J_\disc S_0 \|_{L^2(\O)}^2.
\end{aligned}
\]
The assertion now follows by performing the Cauchy–Schwarz and Young’s inequalities, using the definition of $C_\disc$, and taking the supremum on $m=\in\{0,...,N\}$, respectively.
\end{proof}

\begin{definition}
The semi--norms $|\cdot|_{\star,\disc}$ and $|\cdot|_{\star,\widetilde\disc}$ are respectively defined on $\L^2(\O)$ and $L^2(\O)$ by
\begin{equation}\label{eq-def-norm-1}
|\w|_{\star,\disc}:=\dsp\sup
\Big\{
\w(\x)\cdot\Pi_\disc \v(\x) \ud \x \; : \; \v\in E_\disc \mbox{ such that $\|\v\|_\disc=1$}
\Big\},
\end{equation}
and
\begin{equation}\label{eq-def-norm-2}
|w|_{\star,\widetilde\disc}:=\dsp\sup
\Big\{
w(\x)\cdot\widetilde\Pi_\disc r(\x) \ud \x \; : \; r\in \Y_{\disc,0} \mbox{ such that $\|r\|_\disc=1$}
\Big\},
\end{equation}
where $E_\disc=\{\v\in \X_{\disc,0}\;:\; \div_\disc \v =0\}$.
\end{definition}

\begin{lemma}
\label{lemma-est-dual}
Assume \eqref{assump-nsp} holds. If $(\u,p,S)$ is a discrete solution to \eqref{nsp-gs}, then there exists constants $C_3,C_4>0$, such that
\begin{equation}\label{eq-est-dual}  
\dsp\int_0^T \| \delta_\disc \u \|_{\star,\disc} \ud t \leq C_3
\mbox{ and }
\dsp\int_0^T \| \widetilde\delta_\disc S \|_{\star,\widetilde\disc} \ud t \leq C_4.
\end{equation}
\end{lemma}

\begin{proof}
With any $\v\in E_\disc$ and $r\in \Y_{\disc,0}$, the scheme \eqref{nsp-gs} can be rewritten as follows
\[
\begin{aligned}
\dsp\int_\O \delta_\disc^{(n+\frac{1}{2})}\u(\x) \Pi_\disc \v(\x) \ud \x
&\leq a_2 \| \nabla_\disc \u^{(n+1)} \|_{\L^2(\O\times(0,T))^d} 
\| \nabla_\disc \v \|_{\L^2(\O\times(0,T))^d}  \\
&\quad+C_\disc\|\g\|_{\L^2(\O)}\| \nabla_\disc \v\|_{\L^2(\O\times(0,T))^d} \|,
\end{aligned}
\]
and
\[
\begin{aligned}
\dsp\int_\O \widetilde\delta_\disc^{(n+\frac{1}{2})}S(\x) \Pi_\disc r(\x) \ud \x
&\leq \mu \| \widetilde\nabla_\disc S^{(n+1)} \|_{\L^2(\O\times(0,T))^d} \| 
\| \widetilde\nabla_\disc r \|_{\L^2(\O\times(0,T))^d} \| \\
&\quad+C_\disc\|h\|_{L^2(\O)}\| \widetilde\nabla_\disc r\|_{L^2(\O\times(0,T))^d} \|.
\end{aligned}
\]
The assertion then follows by taking the supremum over $\v\in E_\disc$ with $\|\v\|_\disc=1$ (resp. over $r\in \Y_{\disc,0}$ with $\|r\|_{\widetilde\disc}=1$ ), multiplying by $\delta t^{(n+1)}$, summing over $n\in\{0,...,N-1\}$, and using  \eqref{eq-est-sol-1} and \eqref{eq-est-sol-2}.

\end{proof}

We now state and prove the following result concerning the convergence of the discrete scheme.
\begin{theorem}
Assume that Assumptions \eqref{assump-nsp} are satisfied. If $(\disc_{T_m})_{m\in\NN}$ is a sequence of a space-time gradient discretisation, that is coercive, consistent, limit conforming, trilinear limit conforming, and compact, and $(\u_m,p_m,S_m)$ is a solution to the approximate problem \eqref{nsp-gs} for any $m\in\NN$, then the problem \eqref{nsp-weak} has a solution $(\bar\u,\bar p,\bar S)$, such that, up to a subsequence, as $m \to \infty$, the following convergences hold:
\begin{itemize}
\item $(\Pi_{\disc_m}\u_m,\widetilde\Pi_{\disc_m}S_m)$ converges to $(\bar \u,\bar S)$ in $L^2(0,T;\L^2(\O)) \times L^2(0,T;L^2(\O))$,
\item $(\nabla_{\disc_m}\u_m,\widetilde\nabla_{\disc_m}S_m)$ converges weakly to $(\nabla\bar \u,\nabla\bar S)$ in $\L^2(\O\times(0,T))^d \times L^2(\O\times(0,T))^d$,
\item $\chi_{\disc_m} p_m$ converges weakly to $\bar p$ in $L^2(\O\times(0,T))$.
\end{itemize}
\end{theorem}

\begin{proof}
Thanks to Estimates \eqref{eq-est-sol-1} and \eqref{eq-est-sol-2} and the consistency and the limit–conformity properties, the hypothesis of \cite[Lemma 4.8]{B1} are satisfied. As a consequence, there exists $(\bar\u,\bar S) \in L^2(0,T;\L^2(\O)) \times L^2(0,T;L^2(\O))$ such that $(\Pi_{\disc_m}\u_m,\widetilde\Pi_{\disc_m}S_m)$ converges to $(\bar\u,\bar S)$ weakly in $L^2(0,T;\L^2(\O)) \times L^2(0,T;L^2(\O))$, and $(\nabla_{\disc_m}\u_m,\widetilde\nabla_{\disc_m}S_m)$ converges to $(\nabla\bar\u,\nabla\bar S)$ weakly in $\L^2(\O\times((0,T))^d \times L^2(\O\times((0,T))^d$. On the other hand, \cite[Theorem 4.14]{B1} provides the strong convergence of $(\Pi_{\disc_m}\u_m,\widetilde\Pi_{\disc_m}S_m)$ to $(\bar\u,\bar S)$ in $L^\infty(0,T;\L^2(\O)) \times L^\infty(0,T;L^2(\O))$, thanks to Estimate \eqref{eq-est-dual} to the three properties; the consistency, limit--conformity and the compactness.

We now prove that the pair $(\bar \u,\bar S)$ satisfies equations in our continuous problem. To do so, consider $\bar\bpsi \in L^2(0,T;\L^2(\O)^d)$ and $\bar\varphi \in L^2(0,T;L^2(\O)^d)$ such that $\dr_t\bar\bpsi \in \L^2(\O\times(0,T))$ and $\dr_t\bar\varphi \in L^2(\O\times(0,T))$, and $\bar\bpsi(T,\cdot)=0$ and $\bar\varphi(T,\cdot)=0$. Utilizing the interpolation results proposed in \cite{B1},  we can construct $(\w_m,\widetilde w_m) \in (\X_{\disc,0} \times Y_{\disc,0})$ such that the following convergences are fulfilled:$(\Pi_{\disc_m}w_m, \widetilde\Pi_{\disc_m}\widetilde w_m)$ converges to $(\bar\bpsi,\bar\varphi)$ strongly in $L^2(0,T;\L^2(\O)) \times L^2(0,T;L^2(\O))$. $(\delta_{\disc_m}\w_m, \widetilde\delta_{\disc_m}\widetilde w_m)$ converges to $(\dr_t \bar\bpsi,\dr_t \bar\varphi)$ strongly in $\L^2(\O\times(0,T)) \times L^2(\O\times(0,T))$.
Inserting $\v_m^{(n+1)}:=\delta t_m^{(n+\half)}\w_m$ and $r_m^{(n+1)}:=\delta t_m^{(n+\half)}\widetilde w_m$ in the scheme \eqref{nsp-gs}, we obtain
\begin{equation}\label{nsp-gs-1-proof}
\begin{aligned}
&\dsp\sum_{n=0}^{N_m-1}\dsp\int_\O [\Pi_{\disc_m}\u_m^{(n+1)}-\Pi_{\disc_m}\u^{(n)}] \cdot \Pi_{\disc_m} \w_m \ud \x\\
&\quad+\dsp\int_0^T\dsp\int_\O V(\widetilde{\Pi}_{\disc_m}S_m)\nabla_{\disc_m} \u_m: \nabla_{\disc_m}\w_m \ud \x \ud t\\
&\quad+\dsp\int_0^T\cA_{\disc_m}(\u_m,\w_m) \ud t
\quad-\dsp\int_0^T\dsp\int_\O({\div}_{\disc_m} \w) \chi_{\disc_m} p_m \ud \x \ud t\\
&=\dsp\int_0^T\dsp\int_\O \g \Pi_{\disc_m} \w \ud \x \ud t,
\end{aligned}
\end{equation}
\begin{equation}\label{nsp-gs-2-proof}
\dsp\int_\O {\div}_\disc \u^{(n+1)} \chi_\disc q \ud x =0, 
\end{equation}
\begin{equation}\label{nsp-gs-3-proof}
\begin{aligned}
&\dsp\sum_{n=0}^{N_m-1}\dsp\int_\O [\widetilde{\Pi}_{\disc_m}S_m^{(n+1)}-\widetilde{\Pi}_{\disc_m}S^{(n)}] \cdot \widetilde{\Pi}_{\disc_m} \widetilde w_m\ud \x
+\mu \dsp\int_0^T\dsp\int_\O \widetilde{\nabla}_{\disc_m} S_m\cdot \widetilde{\nabla}_{\disc_m} \widetilde w_m \ud \x \ud t\\
&\quad+\dsp\int_0^T\cB_{\disc_m}(\u_m,S_m,\widetilde w_m)\\
&=\dsp\int_0^T\dsp\int_\O h\widetilde{\Pi}_{\disc_m} \widetilde w_m \ud \x \ud t,
\end{aligned}
\end{equation}
Note that $\w^{(N)}=0$ and $\widetilde w^{(N)}=0$. Performing the discrete integration by the part formula \cite[Equation (D.15)]{B1} to the above equality implies
\begin{equation}\label{nsp-gs-4-proof}
\begin{aligned}
&-\dsp\int_0^T\dsp\int_\O \Pi_{\disc_m}\u_m\cdot \delta_{\disc_m} \w_m \ud \x \ud t
-\dsp\int_\O \Pi_{\disc_m}\u_m^{(0)}\cdot \delta_{\disc_m} \w_m^{(0)} \ud \x\\
&\quad+\dsp\int_0^T\dsp\int_\O V(\widetilde{\Pi}_{\disc_m}S_m)\nabla_{\disc_m} \u_m: \nabla_{\disc_m}\w_m \ud \x \ud t\\
&\quad+\dsp\int_0^T\cA_{\disc_m}(\u_m,\w_m) \ud t
-\dsp\int_0^T\dsp\int_\O({\div}_{\disc_m} \w) \chi_{\disc_m} p_m \ud \x \ud t\\
&=\dsp\int_0^T\dsp\int_\O \g \Pi_{\disc_m} \w \ud \x \ud t,
\end{aligned}
\end{equation}
\begin{equation}\label{nsp-gs-5-proof}
\dsp\int_\O {\div}_\disc \u^{(n+1)} \chi_\disc q \ud x =0, 
\end{equation}
\begin{equation}\label{nsp-gs-6-proof}
\begin{aligned}
&\quad-\dsp\int_0^T\dsp\int_\O \widetilde{\Pi}_{\disc_m}S_m \widetilde{\delta}_{\disc_m} \widetilde{w}_m \ud \x \ud t
-\dsp\int_\O \widetilde{\Pi}_{\disc_m}S_m^{(0)}\widetilde{\delta}_{\disc_m}\widetilde{w}_m^{(0)} \ud \x\\
&\quad+\mu \dsp\int_0^T\dsp\int_\O \widetilde{\nabla}_{\disc_m} S_m\cdot \widetilde{\nabla}_{\disc_m} \widetilde w_m \ud \x \ud t
+\dsp\int_0^T\cB_{\disc_m}(\u_m,S_m,\widetilde w_m)\\
&=\dsp\int_0^T\dsp\int_\O h\widetilde{\Pi}_{\disc_m} \widetilde w_m \ud \x \ud t,
\end{aligned}
\end{equation}
it follows directly from the consistency proerty that $(\Pi_{\disc_m}\u_m^{(0)},\widetilde{\Pi}_{\disc_m}S_m^{(0)})$ converges to $(\u_0,S_0)$ strongly in $(\L^2(\O) \times L^2(\O))$. Now, the strong convergence of $\Pi_{\disc_m}\u_m$ and $\widetilde{\Pi}_{\disc_m}S_m$ established here, and the assumptions inforced on $V$ enable us to apply the dominated convergence theorem. The assertion, $(\bar\u,\bar S,\bar p)$ is a continuous solution, follows from passing to the limit in each term of the above representation.
\end{proof}


\section{Numerical Tests}\label{sec-numerical}
In this section, we perform the hybrid finite volume methods to solve the Navier-Stokes problem \eqref{nsp} on the square domain $\O=[-1,1]^2$. The scheme is a kind of polytopal method that preserves the physical properties of models. With a specific choice of the discrete elements given in Definition \ref{def-gd-nsp-T}, we can show that the method can be presented in the gradient scheme format (the scheme \eqref{nsp-gs}), we refer the reader to \cite[Section 4]{Y-NSP-2021} for details. We use two different types of polygonal meshes to partition the spatial domain, as displayed in Figure \ref{fig-mesh}. A backward Euler discretisation is employed for the time step, and the test is performed at $T=0.1$.  

\begin{figure}[ht]\label{fig-mesh}
	\begin{center}
	\begin{tabular}{cc}
	\includegraphics[width=0.40\linewidth]{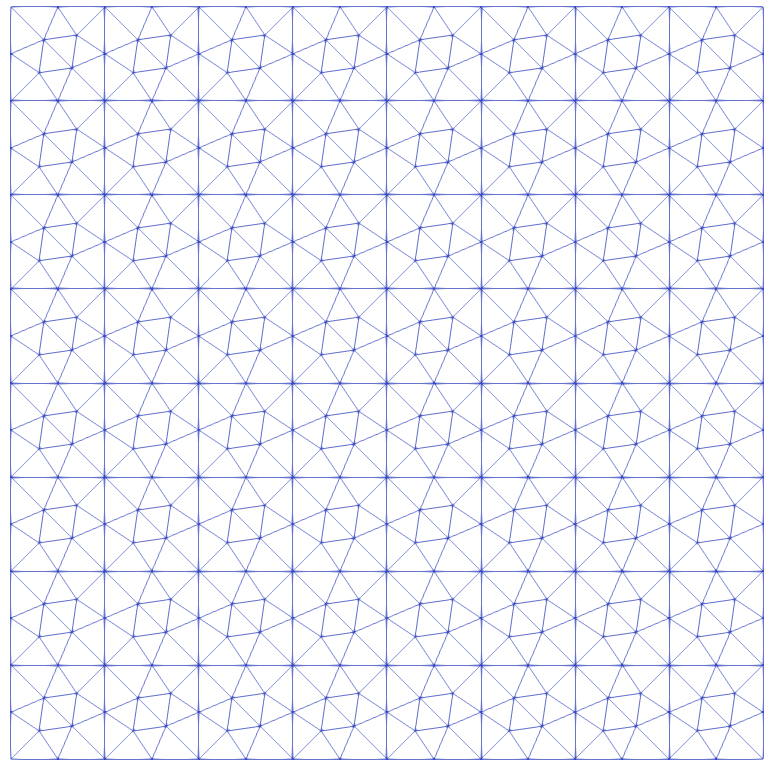} & \includegraphics[width=0.40\linewidth]{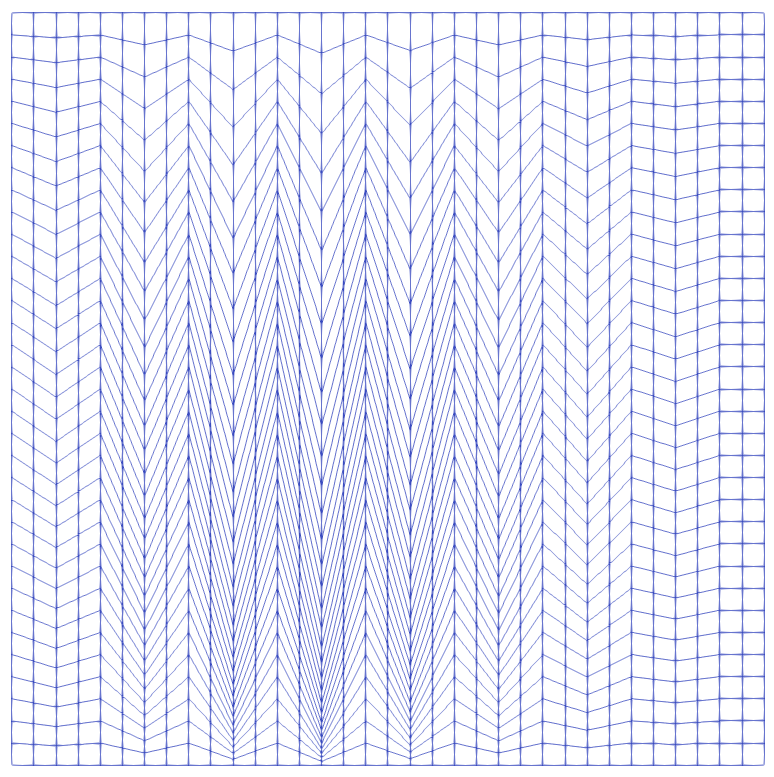}\\
	 \texttt{Triangular}  & \texttt{Distroted}
	\end{tabular}
	\end{center}
	\caption{Sample of the polygonal meshes.}
	\label{fig-1}
\end{figure}

In order to investigate the quality of the hybrid finite volume methods, we measure the errors as the difference between the exact solution and the discrete one in a suitable norm. More precisely, we use the following quantities:
\begin{equation}\label{eq-error}
E_1:=\frac{\| \bar \u(\cdot) - \Pi_\disc \u\|_{\L^{2}(\O)}}{\| \bar \u(\cdot)\|_{\L^{2}(\O)}},\;
E_2:=\frac{\| \bar p(\cdot) - \chi_\disc p\|_{L^{2}(\O)}}{\| \bar p(\cdot)\|_{L^{2}(\O)}},\;
\mbox{ and }
E_3:=\frac{\| \bar S(\cdot) - \widetilde{\Pi}_\disc S\|_{L^{2}(\O)}}{\| \bar S(\cdot)\|_{L^{2}(\O)}}.
\end{equation}

First, we perform the test on the triangular mesh type to examine the accuracy of the scheme where the Navier-Stokes equations are not dependent on the temperature ($V(S)=1$). The exact solution is
\begin{equation}\label{exact-u}
\bar\u(x,y):=
\begin{pmatrix}
\bar u_1\\
\bar u_2
\end{pmatrix}
=
\begin{pmatrix}
\sin((\pi + t)y) \cos((\pi + t)x)\\
\cos((\pi + t)y) \sin((\pi + t)x)
\end{pmatrix},
\end{equation}
\begin{equation}\label{exact-p}
\bar p(x, y) = \sin((\pi + t)x) \cos((\pi + t)y),
\end{equation}
\begin{equation}\label{exact-S}
\bar S(x, y) = t \sin(x + y),
\end{equation}

We post in Figure \ref{fig-1} and Figure \ref{fig-2} the convergence rate graphs of the velocity, pressure, and temperature errors. The convergence rates are close to one, which is acceptable compared to well-known $\mathcal O(h)$-estimates. Furthermore, we make use of the least squares residual test. The resultant errors can be presented in the following formula:
\[
E=Ch^r,
\]
where $C$ and $r$ are positive constants, $h$ is the mesh size, and $E$ stands to errors computed by \eqref{eq-error}. Therefore, it yields
\[
\log(E_i)=\log(C)+r\log(h),\; i=1,2,3.
\]
We observe that the slopes are very close to one with the least squares residual of $0.21$, $0.20$, $0.11$, and $0.08$ with respect to errors on $\bar \u=(\bar u_1,\bar u_2)$, $\bar p$ and $\bar S$, respectively. 
\begin{figure}[ht]
	\begin{center}
	\includegraphics[scale=0.4]{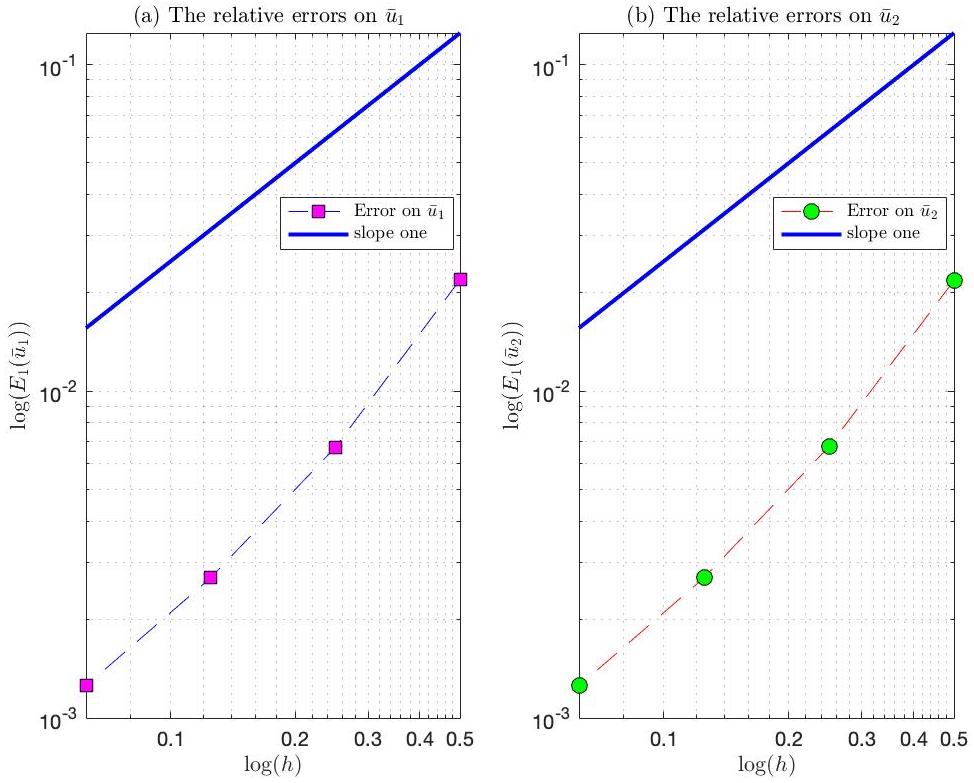}
	\end{center}
\caption{The relative errors for the case $V(S)=1$.}	
\label{fig-1}
\end{figure}

\begin{figure}[ht]
	\begin{center}
	\includegraphics[scale=0.4]{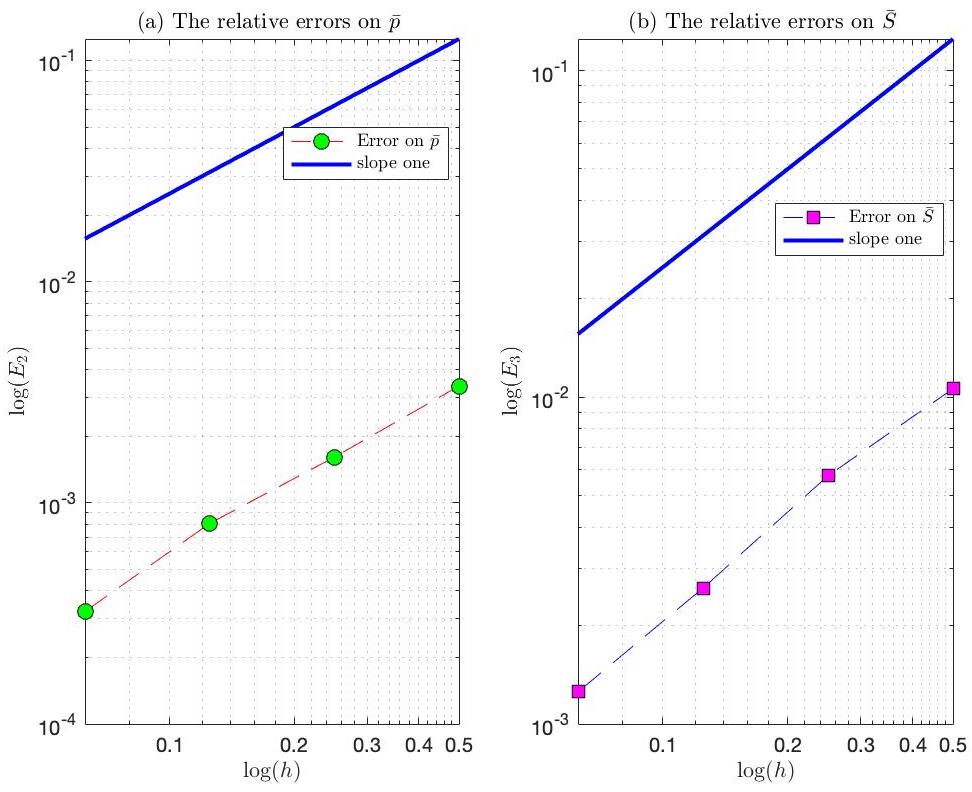}
	\end{center}
\caption{The relative errors for the case $V(S)=1$.}
\label{fig-2}
\end{figure}

Second, one vital feature of the hybrid finite volume method (used here) is the flexibility in dealing with complex geometries distorting cells. Using the distorted mesh, we examine the effectiveness of the scheme when the viscosity $V$ depends on the temperature $S$ in the problem \eqref{nsp}, i.e., $V(S)=\sqrt{S^2+1}+2$. The analytical solution is still given by \eqref{exact-u}--\eqref{exact-p}. We display the convergence rates in Figure \ref{fig-3} and Figure \ref{fig-4}, and notice that the scheme is still rigorous concerning the non-linearity and the mesh distortion, in which the slopes are one with the least squares residual of $0.05$, $0.04$. $0.11$, and $0.09$ respect with $L_2$-error on $\bar \u=(\bar u_1, \bar u_2)$, $\bar p$ and $\bar S$, respectively.    

\begin{figure}[ht]
	\begin{center}
	\includegraphics[scale=0.4]{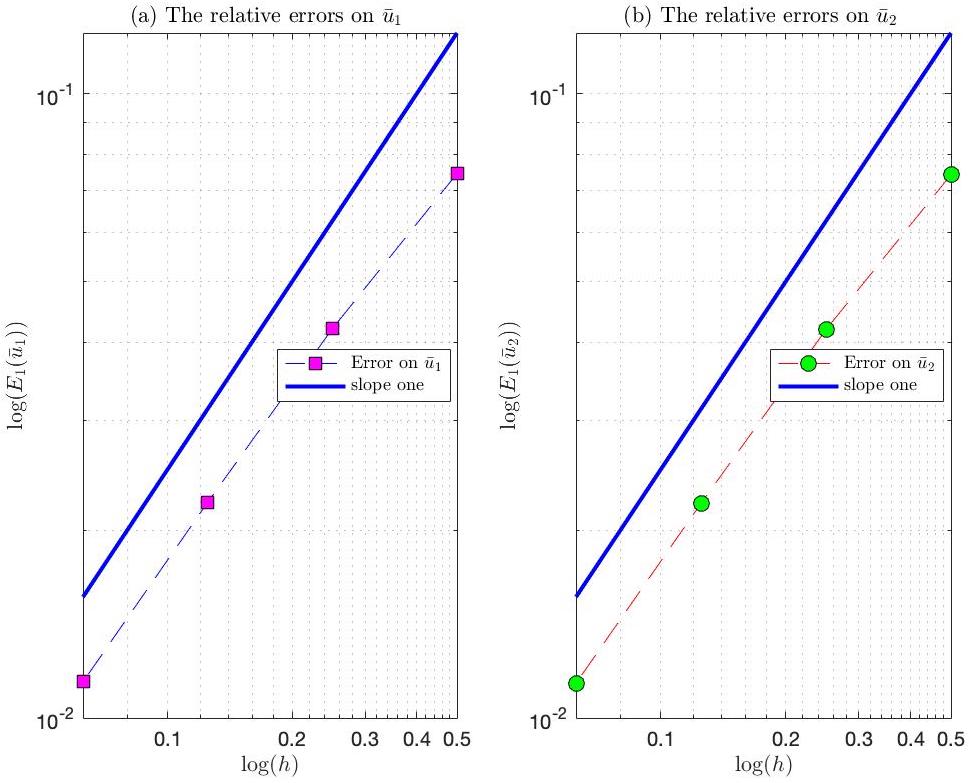}
	\end{center}
	\caption{The relative errors for the case $V(S)=\sqrt{S^2+1}+2$.}
\label{fig-3}
\end{figure}

\begin{figure}[ht]
	\begin{center}
	\includegraphics[scale=0.4]{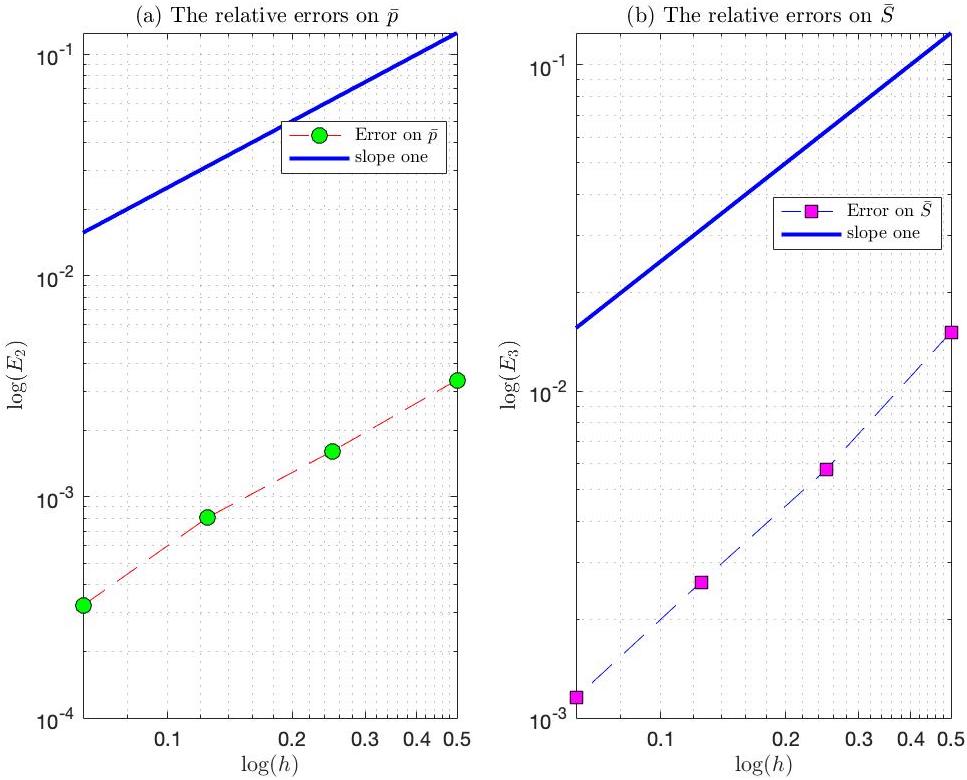}
	\end{center}
\caption{The relative errors for the case $V(S)=\sqrt{S^2+1}+2$.}
\label{fig-4}
\end{figure}



\bibliographystyle{siam}
\bibliography{PNSP-ref}

\end{document}